\documentclass[a4paper]{article}
\usepackage{amsthm,amsmath,amssymb}
\usepackage{enumerate,enumitem}
\usepackage{graphicx}
\usepackage{xfrac}
\usepackage{ulem}
\usepackage{mathtools}
\usepackage{ wasysym }
\usepackage{float}
\usepackage{tabularx,booktabs,multirow}
\usepackage{hyperref}
\newcolumntype{C}{>{\centering\arraybackslash}X}
\newcolumntype{D}{>{\centering\arraybackslash}X}

\DeclareGraphicsExtensions{.pdf,.png,.eps}
\graphicspath{{figs/}}

\newtheorem{theorem}{Theorem}
\newtheorem{lemma}[theorem]{Lemma}

\newtheorem{proposition}[theorem]{Proposition}
\newtheorem{corollary}[theorem]{Corollary}

\newtheorem{conjecture}[theorem]{Conjecture}

\newtheorem*{claim*}{Claim}

\theoremstyle{remark}

\newcommand{\Vol}{{\rm Vol}}

\newcommand{\cA}{\ensuremath{\mathcal{A}}}

\newcommand{\cB}{\ensuremath{\mathcal{B}}}
\newcommand{\cD}{\ensuremath{\mathcal{D}}}

\newcommand{\cP}{\ensuremath{\mathcal{P}}}

\newcommand{\cS}{\ensuremath{\mathcal{S}}}

\newcommand{\cT}{\ensuremath{\mathcal{T}}}
\newcommand{\N}{\ensuremath{\mathbb{N}}}

\newcommand{\R}{\ensuremath{\mathbb{R}}}

\newcommand{\qn}{\ensuremath{\lfloor qn \rfloor}}
\newcommand{\QN}{\ensuremath{\lceil qn \rceil}}

\newcommand{\bU}{\ensuremath{\mathbf{U}}}
\newcommand{\bX}{\ensuremath{\mathbf{X}}}
\newcommand{\bY}{\ensuremath{\mathbf{Y}}}
\newcommand{\bZ}{\ensuremath{\mathbf{Z}}}

\newcommand{\QQ}{\ensuremath{\mathcal{Q}}}

\newcommand{\Rw}{\ensuremath{R^{\text{\normalfont w}}}}

\usepackage[usenames,dvipsnames]{xcolor}

\newcommand{\footremember}[2]{%
  \footnote{#2}
  \newcounter{#1}
  \setcounter{#1}{\value{footnote}}%
}
\newcommand{\footrecall}[1]{%
  \footnotemark[\value{#1}]%
}

\begin{document}

\title{Diagonal poset Ramsey numbers}
\author{Maria Axenovich \footremember{grant}{Karlsruhe Institute of Technology, Karlsruhe, Germany. The research was partially supported by DFG grant FKZ AX 93/2-1.}
\and Christian Winter \footrecall{grant}\ \footnote{E-mail: \textit{christian.winter@kit.edu}} }
\maketitle

\begin{abstract}
A poset $(Q,\le_Q)$ contains an induced copy of a poset $(P,\le_P)$ if there exists an injective mapping $\phi\colon P\to Q$ such that for any two elements $X,Y\in P$, $X\le_P Y$ if and only if $\phi(X)\le_Q \phi(Y)$. By $Q_n$ we denote the Boolean lattice $(2^{[n]},\subseteq)$. The poset Ramsey number $R(P,Q)$ for posets $P$ and $Q$ is the least integer $N$ for which any coloring of the elements of $Q_N$ in blue and red contains either a blue induced copy of $P$ or a red induced copy of $Q$.

In this paper, we show that $R(Q_m,Q_n)\le nm-\big(1-o(1)\big)n\log m$ where $n\ge m$ and $m$ is sufficiently large. This improves the best known upper bound on $R(Q_n,Q_n)$ from $n^2-n+2$ to $n^2-\big(1-o(1)\big) n\log n$. Furthermore, we determine $R(P,P)$ where $P$ is an $n$-fork or $n$-diamond up to an additive constant of $2$. 

A poset $(Q,\le_Q)$ contains a weak copy of $(P,\le_P)$ if there is an injection $\psi\colon P\to Q$ such that $\psi(X)\le_Q \psi(Y)$ for any $X,Y\in P$ with $X\le_P Y$.
The weak poset Ramsey number $\Rw(P,Q)$ is the smallest $N$ for which any blue/red-coloring of $Q_N$ contains a blue weak copy of $P$ or a red weak copy of $Q$.
We show that $\Rw(Q_n,Q_n)\le 0.96n^2$.
\end{abstract}

\section{Introduction}

A \textit{poset} is a set $P$ equipped with a partial order $\le_P$, i.e., a binary, transitive, reflexive, and antisymmetric relation. 
Usually, we refer to such a poset $(P,\le_P)$ just as $P$ and call the elements of $P$ \textit{vertices}.
Let $P$ and $Q$ be two posets. A \textit{(strong) embedding} $\phi: P\to Q$ of $P$ into $Q$ is an injective function where for any two $X,Y\in P$, $X\le_P Y$ if and only if $\phi(X)\le_Q \phi(Y)$. We say that the image of $\phi$ is an \textit{(induced) copy} of $P$ in $Q$.
Given a non-empty set $Z$, we denote by $\QQ(Z)=(2^Z,\subseteq)$ the \textit{Boolean lattice} on ground set $Z$ with \textit{dimension} is $|Z|$.
We use $Q_n$ to denote a Boolean lattice with an arbitrary $n$-element ground set. For $\ell\in\{0,\dots,|Z|\}$, the \textit{$\ell$-th layer} of $\QQ(Z)$ is the set $\{X\in\QQ(Z):~ |X|=\ell\}$.
\\

Ramsey-type problems in extremal combinatorics are widely studied for various graph structures. 
Initiated by Ne\v{s}et\v{r}il and R\"odl \cite{NR}, such questions were also considered for posets.
A \textit{blue/red coloring} of a poset $Q$ is a mapping $c\colon Q\to\{\text{blue},\text{red}\}$. We say that a poset is \textit{blue} if every of its vertices is colored blue, and \textit{red} if every of its vertices is colored red.
For an integer $k$ and a poset parameter $p$ such as size, height, or width, Kierstead and Trotter \cite{KT} analysed the minimal $p(Q)$ of a poset $Q$ such that any blue/red coloring $c\colon Q\to\{\text{blue},\text{red}\}$ contains a monochromatic induced copy of a poset $P$ with $p(P)=k$.
When $p(Q)$ is the dimension of the smallest Boolean lattice containing a copy of $Q$, this problem leads to the \textit{poset Ramsey number} introduced by Axenovich and Walzer \cite{AW}.
For posets $P$ and $Q$, let 
\begin{multline*}
R(P,Q)=\min\{N\in\N \colon \text{ every blue/red coloring of $Q_N$ contains either}\\ 
\text{a blue induced copy of $P$ or a red induced copy of $Q$}\}.
\end{multline*}
This extremal parameter has been actively studied in recent years with the central question being the poset Ramsey number $R(Q_n,Q_n)$ for large $n$.
In this setting, first bounds were given by Axenovich and Walzer \cite{AW} who showed that $2n\le R(Q_n,Q_n)\le n^2+2n$.
The upper bound was  improved by Walzer \cite{Walzer} to $n^2+1$, and then by Lu and Thompson \cite{LT} to the best known value $n^2-n+2$.
Cox and Stolee \cite{CS} improved the lower bound to $2n+1$ for $n\ge 13$, which was extended to all $n\ge 3$ by Bohman and Peng \cite{BP}.
The best known lower bound for large $n$ is given by Winter \cite{QnEH}, who showed that $R(Q_n,Q_n)\ge 2.24n$.

Note that any posets $P$ and $Q$ are contained as a copy in $Q_n$ for sufficiently large $n$, therefore $R(P,Q)$ is well-defined.
In the off-diagonal setting, Lu and Thompson \cite{LT} showed that for $n\ge m\ge 4$,
$$R(Q_m,Q_n)\le n\Big(m-2+O\big(\tfrac{1}{m}\big)\Big)+m+3.$$
Further bounds are known if one or two of the parameters $m$ and $n$ are small, e.g. $R(Q_2,Q_n)=n+\Theta\big(\tfrac{n}{\log n}\big)$, where the upper bound is given by Grosz, Methuku and Tompkins \cite{GMT} and the lower bound is by the authors \cite{QnV}. 
Furthermore, $R(Q_2,Q_2)=4$, $R(Q_2,Q_3)=5$, and $R(Q_3,Q_3)=7$, see \cite{AW},\cite{LT}, and Falgas-Ravry et al. \cite{FMTZ}, respectively.
\\

The first result of this paper is a strengthened upper bound on the poset Ramsey number of $Q_m$ versus $Q_n$ when both $m$ and $n$ large.
\begin{theorem}\label{thm:QmQn}
Let $n,m\in\N$ with $2^{25}\le m \le n$. Then $$R(Q_m,Q_n)\le n\left(m-\big(1-\tfrac{2}{\sqrt{\log m}}\big)\log m\right).$$
Moreover, if $n\ge m$ and $\varepsilon\in\R$, $0<\varepsilon<1$, such that $\tfrac{n+m}{n}\cdot \tfrac{1}{(1-\varepsilon)\log m}+m^{-\varepsilon}\le \varepsilon$, then
$$R(Q_m,Q_n)\le n\big(m-(1-\varepsilon)^2\log m\big).$$
\end{theorem}

In this article, $\log$ refers to the logarithm with base $2$.
Theorem \ref{thm:QmQn} is the first result which improves the initial basic upper bound by Axenovich and Walzer \cite{AW}, see Lemma \ref{lem:blob_restated}, by a superlinear additive term.
Our result immediately provides an improved upper bound on $R(Q_n,Q_n)$.

\begin{corollary}\label{cor:QnQnUB}
For $\varepsilon>0$ and sufficiently large $n\in\N$ depending on $\varepsilon$, $R(Q_n,Q_n)\le n^2- (1-\varepsilon)n\log n$.
\end{corollary}

\bigskip

The diagonal poset Ramsey number $R(P,P)$ is known exactly for some basic posets $P$.
Walzer \cite{Walzer} determined this number for antichains and chains.
Chen, Chen, Cheng, Li, and Liu \cite{CCCLL} found $R(P,P)$ when $P$ is the poset which consists of two elementwise incomparable chains on a given number of vertices, with an added vertex smaller than all other vertices. 

Here, we study the diagonal poset Ramsey number $R(P,P)$ for further posets $P$.
The \textit{$n$-fork} $V_n$ is the poset consisting of an antichain on $n$ vertices with an added vertex smaller than all other vertices.
Similarly, the \textit{$n$-diamond} $D_n$ is the poset consisting of an antichain on $n$ vertices and a vertex smaller than all others as well as a vertex larger than all others.
\\

Let $n\in\N$. We denote by $\alpha(n)$ the minimal dimension $N$ such that $Q_N$ contains an antichain of size $n$. We call $\alpha(n)$ the \textit{Sperner number} of $n$.
Sperner \cite{Sperner} showed that $\alpha(n)$ is the minimal integer such that $\binom{N}{\lfloor N/2\rfloor}\ge n$.
It is a basic observation that $\alpha(n)\le \alpha(2n-1)\le \alpha(n)+2$, which we will use repeatedly.
Note that $\alpha(n)=\big(1+o(1)\big)\log n$, see also an almost exact bound on $\alpha(n)$ by Habib, Nourine, Raynaud and Thierry \cite{HNRT}.

\begin{theorem}\label{thm:DnDn}
For every $n\in\N$, $2\alpha(n)\le R(D_n,D_n)\le  \alpha(n)+\alpha(2n-1)$. 
In particular, $2\alpha(n)\le R(D_n,D_n)\le 2\alpha(n)+2$  and thus $R(D_n,D_n)=\big(2+o(1)\big)\log n$.
\end{theorem}

Note that  for some values of $n$, $\alpha(2n-1)\le \alpha(n)+1$, in which case the above upper and lower bounds on $R(Q_n,Q_n)$ differ by $1$.\\

For the next result, we need two further extremal parameters. Given $n,N\in\N$ with $N\ge \alpha(n)$, let $\beta(N,n)$ and $N^*(n)$ be integers with
$$\beta(N,n)=\min\big\{\beta: ~ \binom{N}{\beta}\ge n\big\}\ \text{ and }\ N^*(n)=\max\big\{N\ge\alpha(n): ~ N-\beta(N,n)< \alpha(n)\big\},$$
as illustrated in Figure \ref{fig:Nstar}. 
Note that $\binom{N}{\lfloor \alpha(n)/2\rfloor}\ge \binom{\alpha(n)}{\lfloor \alpha(n)/2\rfloor}\ge n > \binom{N}{0}$, so $1\le \beta(N,n) \le \alpha(n)/2$.
Thus, $\alpha(n)\le N^*(n)< \beta(N^*(n),n)+\alpha(n)\le 2\alpha(n)$, so in particular $\beta(N,n)$ and $N^*(n)$ are well-defined.

\begin{theorem}\label{thm:VnVn}
For every $n\in\N$, $N^*(n)+1\le R(V_n,V_n)\le N^*(n)+3$.\\
Let $d=\frac{1}{1-c}$ where $c$ is the unique real solution of $\log\big(c^{-c}(1-c)^{c-1}\big)=1-c$, i.e., $d\approx 1.29$.
Then $R(V_n,V_n)=(d+o(1))\log(n).$
\end{theorem}

Similarly to Theorem \ref{thm:DnDn}, by using more careful estimates our proof provides $R(V_n,V_n)\le N^*(n)+2$ whenever $\alpha(2n-1)\le \alpha(n)+1$.
\\

The poset Ramsey number $R(P,Q)$ concerns induced copies of posets $P$ and $Q$. A variation of this Ramsey-type problem deals with so called \textit{weak copies}.
Let $P$ and $Q$ be two posets. A \textit{weak embedding} $\psi\colon P\to Q$ is an injective function such that for any two elements $X,Y\in P$,
$X\le_P Y$ implies $\psi(X)\le_Q\psi(Y)$. In this setting, we allow that $\psi(X)\le_Q \psi(Y)$ even if $X\not \le_P Y$.
We say that the image of $\psi$ is a \textit{weak copy} of $P$ in $Q$.
Analogously to the poset Ramsey number $R(P,Q)$, we define the \textit{weak poset Ramsey number} for posets $P$ and $Q$ as
\begin{multline*}
\Rw (P,Q)=\min\{N\in\N \colon \text{ every blue/red coloring of $Q_N$ contains either}\\ 
\text{a blue weak copy of $P$ or a red weak copy of $Q$}\}.
\end{multline*}

It is a basic observation that $\Rw (P,Q)\le R(P,Q)$ for any posets $P$ and $Q$.
The best known bounds in the diagonal setting $Q_n$ versus $Q_n$ are a lower bound by Cox and Stolee \cite{CS} and an upper bound by Lu and Thompson \cite{LT},
$$2n+1\le \Rw(Q_n,Q_n)\le R(Q_n,Q_n)\le n^2-n+2.$$
Gr\'osz, Methuku and Tompkins \cite{GMT} showed in the off-diagonal setting that
$\Rw (Q_m,Q_n)\ge m+n+1$ for $m\ge 2$ and $n\ge 68$, and $\Rw(Q_m,Q_n)\le n+2^m-1$, where the second bound is derived from a result by Cox and Stolee \cite{CS}.
\\

Our final result is an improvement of the upper bound on $\Rw(Q_n,Q_n)$. 

\begin{theorem}\label{thm:weakQnQn}
For sufficiently large $n$, $\Rw (Q_n,Q_n)\le 0.96n^2$.
\end{theorem}


In Section \ref{sec:QnQn_prelim}, we introduce some notation and definitions, and discuss the \textit{Blob Lemmas}.
In Section \ref{sec:QmQn}, we present a proof of Theorem \ref{thm:QmQn}. 
In Section \ref{sec:DnVn}, we prove Theorems \ref{thm:DnDn} and \ref{thm:VnVn}.
A proof of Theorem \ref{thm:weakQnQn} is given in Section \ref{sec:weakQnQn}.
We omit floors and ceilings where appropriate. We denote the set $\{1,\dots,n\}$ of the first $n$ natural numbers by $[n]$.
%

\section{Notation and preliminary results}\label{sec:QnQn_prelim}

\subsection{Notation and definitions}\label{sec:QmQn_def}

In this paper we use capital letters to denote subsets of the ground set in a Boolean lattice, i.e., the vertices of Boolean lattices. 
We use bold capital letters for significant ground sets. 
Calligraphic capital letters are used for families of sets and thus for subposets of the Boolean lattice.
Abstract posets which are not considered as subposets of a specific Boolean lattice are denoted by capital letters, for example $Q_n$, $D_n$ and $V_n$.
\\

Given a set $\bZ$ with disjoint subsets $S,T\subseteq \bZ$, we define a \textit{blob} in a Boolean lattice $\QQ(\bZ)$ as 
$\cB(S;T)=\{Z\subseteq \bZ: S\subseteq Z\subseteq S\cup T\}.$
We call $T$ the \textit{variable set} of this blob.
Note that $\cB(S;T)$ is a Boolean lattice of dimension $|T|$. We say that $|T|$ is the {\it dimension} of the blob. 
We remark that every $Z\in\cB(S;T)$ has the form $S\cup T_Z$ where $T_Z\subseteq T$.
A $t$-{\it truncated blob}, denoted $\cB(S; T; t)$, is the poset $\{ Z \in \cB(S; T): ~ |Z| \leq t\}$. We also say that $\cB(S;T; t)$ has \textit{dimension} $|T|$.
Given a Boolean lattice $\QQ(\bX)$ on ground set $\bX$ and a non-negative integer $t$, $t\leq |\bX|$,  let $\QQ(\bX)^{t}$ denote the \textit{$t$-truncated Boolean lattice}, 
that is the subposet $\{Z\in\QQ(\bX):~ |Z|\le t\}=\cB(\varnothing; \bX; t)$. 
Given two non-negative integers $s$ and $t$, $0\leq s\leq t  \leq |\bX|$,   let $\QQ(\bX)_{s}^t$ denote the \textit{$(s,t)$-truncated Boolean lattice}, 
that is the subposet $\{Z\in\QQ(\bX):~  s \leq |Z|\le t\}$. In particular, $\QQ(\bX)^{t} = \QQ(\bX)^t_0$. 
\\

Let $\bX$ and $\bY$ be disjoint, non-empty sets. Let $\phi\colon \QQ(\bX)^t\to \QQ(\bX\cup\bY)$ be a (strong) embedding.
The function $\phi$ is \textit{$\bX$-good} if $\phi(X)\cap\bX=X$ for every $X\in\QQ(\bX)^{t}$. Note that a $t$-truncated blob $\cB(S; T; t)$ is the image of a $T$-good embedding of $\QQ(T)^{t}$. 
We say that $\phi$ is \textit{red} if its image is a red poset, i.e.\ $\phi$ maps only to red vertices, and \textit{blue} if its image is a blue poset.
If $\phi$ is an embedding of a poset $\cA$ into a Boolean lattice, we use the notation  $\phi(\cA)$ for $\{\phi(X): ~X\in \cA\}$. 
For a subposet $\cA$ of a Boolean lattice, we say that the {\it volume} of  $\cA$, denoted $\Vol(\cA)$, is the total number of ground elements in all vertices of $\cA$, i.e., $\Vol(\cA) = |\bigcup_{X\in \cA} X|$.
\\

\subsection{Blob Lemmas}\label{sec:QmQn_blob}
In order to show an upper bound on $R(Q_m,Q_n)$, we have to find a blue copy of $Q_m$ or a red copy of $Q_n$ in every coloring of a host Boolean lattice.
Using our notation, let us briefly reiterate the basic approach used for previously known upper bounds, see Lemma 1 in \cite{KT}, Blob Lemma in \cite{AW}, and Lemma 1 in \cite{LT}.

\begin{lemma}[Blob Lemma] \label{lem:blob_restated}
Let $n,m\in\N$ and $N=nm+n+m$.
Any blue/red coloring of $\QQ([N])$ contains a blue copy of $Q_m$ or a red copy of $Q_n$. 
\end{lemma}
\begin{proof}
Partition $[N]$ arbitrarily into sets $\bX, \bY^{(0)}, \bY^{(1)},\dots, \bY^{(n)}$ such that $|\bX|=n$ and $|\bY^{(i)}|=m$, $i\in\{0,\dots,n\}$.
We construct a red embedding $\phi\colon \QQ(\bX)\to\QQ([N])$.
For each $X\in \QQ(\bX)$, consider the blob $\cB_X=\cB(X\cup \bigcup_{i=0}^{|X|-1} \bY^{(i)} ;\bY^{(|X|)})$, 
where we use the convention that $\bigcup_{i=0}^{-1}\bY^{(i)}=\varnothing$.
If one of the blobs is monochromatic blue, it is a blue copy of $Q_m$ as desired.

Suppose that there is a red vertex $Z_X\in\cB_X$ for every $X$.
Then the function $\phi\colon \QQ(\bX)\to\QQ([N])$, $\phi(X)=Z_X$ has a red image.
Observe that $\phi(X)\cap\bX=X$ for every $X\subseteq\bX$.
Then it is straightforward to verify that $\phi$ is an embedding. Therefore, there is a red copy of $Q_n$.
\end{proof}

The general proof idea of our bound on $R(Q_m,Q_n)$ is to refine Lemma \ref{lem:blob_restated} by 
considering truncated blobs instead of blobs, moreover those chosen based  on already embedded layers.
In addition, we control the volume of truncated blobs while constructing the embedding.
For this we need parts (i),(ii),(iii) of the following variant of Lemma \ref{lem:blob_restated}.
Part (iv) will be applied to achieve an upper bound on $\Rw(Q_n,Q_n)$.

\begin{lemma}[Truncated Blob Lemma]\label{lem:truncatedCompletion}
Let $n,m,t, N\in\N$. Let $\bX\subseteq [N]$.  Fix a blue/red coloring of the Boolean lattice $\QQ([N])$.
\begin{enumerate}
\item[(i)] If $|\bX|=n$, $t\leq n$, and there is an $\bX$-good, red embedding $\phi\colon \QQ(\bX)^t\to \QQ([N])$ such that $\Vol({\phi}(\QQ(\bX)^t) \le N -(n-t)m$,
then $\QQ([N])$ contains a blue copy of $Q_m$ or a red copy of $Q_n$.
\item[(ii)] If $|\bX|=m$, $t\leq m$,  and  there is an $\bX$-good, blue embedding $\phi\colon \QQ(\bX)^t\to \QQ([N])$ such that $\Vol({\phi}(\QQ(\bX)^t) \le N -(m-t)n$,
then $\QQ([N])$ contains a blue copy of $Q_m$ or a red copy of $Q_n$.
\item[(iii)] If $|\bX|=n$, $t\leq n$,  and there is a set $S$ disjoint from $\bX$ and a red truncated blob $\cB(S; \bX; t)$ such that $|S\cup \bX|\leq N-(n-t)m$, then there is a blue copy of $Q_m$ or a red copy of $Q_n$.
\item[(iv)] If  $0\leq s\leq t \leq n$, and $N= (t-s+2)n$, then $\QQ([N])$ contains  a blue copy of $Q_n$ or a red copy of 
$\QQ([n])^{t}_{s}$.
\end{enumerate}

\end{lemma}

\begin{proof}
\textbf{(i):} We shall extend $\phi$ to a red embedding $\phi'\colon \QQ(\bX)\to \QQ([N])$.
As in the proof of Lemma \ref{lem:blob_restated}, we select disjoint sets of ground elements and use them to define a blob for each not yet embedded $X\in\QQ(\bX)$.
Let $\bU = \bigcup_{X\in\QQ(\bX)^t} \phi(X)$ and note that $|\bU|= \Vol(\phi(\QQ(\bX)^t))$.
Since $\phi$ is $\bX$-good,  $\bX\subseteq \bU$. Thus, $$| [N]\setminus (\bU\cup \bX) |=|[N]\setminus \bU|=N - \Vol(\phi(\QQ(\bX)^t))\ge (n-t)m.$$
Fix $n-t$ pairwise disjoint $m$-element subsets $\bY^{(t+1)},\dots,\bY^{(n)}$ of $[N]\setminus \bU$.
For every $X\in\QQ(\bX)$ with $|X|>t$, consider the blob
$$\cB_X=\cB\left(X\cup (\bU\setminus \bX) \cup \bigcup_{i=t+1}^{|X|-1}\bY^{(i)}; \bY^{(|X|)}\right).$$
If $\cB_X$ is blue, it is a blue copy of $Q_m$, so suppose that there is a red vertex $Z_X\in\cB_X$.
Let $\phi'\colon \QQ(\bX)\to \QQ([N])$ with
$$\phi'(X)=\begin{cases}\phi(X) &\text{ if }|X|\le t,\\ Z_X & \text{ if }|X|> t.\end{cases}$$
The image of $\phi'$ is red. Now we verify that $\phi'$ is an embedding. 
Observe that $\phi'(X)\cap\bX=X$ for every $X\subseteq \bX$. Let $X_1, X_2\subseteq \bX$.
If $\phi'(X_1)\subseteq \phi'(X_2)$, then $X_1=\phi'(X_1)\cap\bX\subseteq \phi'(X_2)\cap\bX =X_2$.
Conversely, suppose that $X_1\subseteq X_2$. 
\\
If $|X_1|\le |X_2|\le t$, then $\phi'(X_1)=\phi(X_1)\subseteq \phi(X_2)=\phi'(X_2)$ because $\phi$ is an embedding.
\\
If $|X_1|\le t < |X_2|$, we have $\phi'(X_1)=X_1\cup (\phi(X_1)\setminus\bX)\subseteq X_2 \cup (\bU\setminus\bX) \subseteq Z_{X_2}=\phi'(X_2)$.
\\
If $t<|X_1|\le|X_2|$, then $\phi'(X_1)\in \cB_{X_1}$, so $\phi'(X_1)\subseteq X_1\cup (\bU\setminus\bX) \cup \bY^{(t+1)}\cup\dots\cup \bY^{(|X_1|)}\subseteq Z_{X_2}=\phi'(X_2)$, as desired.\\

\noindent \textbf{(ii):} This part is proven analogously to part (i).\\

\noindent \textbf{(iii):} Observe that $\cB(S; \bX; t)$ is the image of an $\bX$-good embedding $\phi\colon \QQ(\bX)^t\to \QQ(\bX)$, $\phi(X)=S\cup X$, with $\Vol({\phi}(\QQ(\bX)^t)=|S\cup\bX|$.\\

\noindent \textbf{(iv):} 
Let $\bX=[n]$. Choose disjoint, $n$-element sets $\bY^{(s)},\dots,\bY^{(t)}$ of $\bY=[N]\setminus \bX$.
For each $X\in\QQ(\bX)$ with $s\leq |X|\leq t$, we define a blob 
$$\cB_X=\cB\left(X\cup  \bigcup_{i=s}^{|X|-1}\bY^{(i)};\bY^{(|X|)}\right),$$
where $ \bigcup_{i=s}^{|X|-1}\bY^{(i)} = \varnothing$ for $|X|=s$.
If $\cB_X$ is blue, it corresponds to a blue copy of $Q_n$. So let $\phi(X)$ be a red vertex in $\cB_X$.
Then, we see that $\{\phi(X): X\in \bX\}$ is a red copy of $\QQ(\bX)^{t}_{s}$.
\end{proof}

\section{Upper bound on Ramsey number  $R(Q_m,Q_n)$}\label{sec:QmQn}
First, we need the following technical proposition.
\begin{proposition}\label{prop:epsilon}
Let $2^{25}\le m \le n$. Let $\varepsilon=\tfrac{1}{\sqrt{\log m}}$. Then $\tfrac{n+m}{n}\cdot \tfrac{1}{(1-\varepsilon)\log m}+m^{-\varepsilon}\le \varepsilon.$
\end{proposition}
\begin{proof}
The bound $m\ge 2^{25}$ implies that $\varepsilon=\tfrac{1}{\sqrt{\log m}}\le \tfrac{1}{5}$,
so in particular $4\varepsilon\le 1-\varepsilon$.
Since $m\le n$ and $\log m=\varepsilon^{-2}$, we obtain that
$$
\frac{n+m}{n}\cdot\frac{1}{(1-\varepsilon)\log m}\le \frac{2n}{n}\cdot\frac{\varepsilon^2}{1-\varepsilon}
=\frac{\varepsilon}{2}\cdot \frac{4\varepsilon}{1-\varepsilon}
\le \frac{\varepsilon}{2}.
$$
For $\varepsilon\le \tfrac{1}{5}$, it is straightforward to check that $\tfrac{1}{\varepsilon}\ge 1-\log \varepsilon$. 
Thus, using again that $\log m=\varepsilon^{-2}$,
$$m^{-\varepsilon}=2^{-\frac{1}{\varepsilon}}\le 2^{-1+\log \varepsilon}= \frac{\varepsilon}{2}.$$
Therefore,
$$\frac{n+m}{n}\cdot \frac{1}{(1-\varepsilon)\log m}+m^{-\varepsilon}\le \frac{\varepsilon}{2} + \frac{\varepsilon}{2} = \varepsilon.$$
\end{proof}

\noindent Now we show of our main result.

\begin{proof}[Proof of Theorem \ref{thm:QmQn}]
Fix $n$ and $m$ such that $n\ge m$. Fix an $\varepsilon\in\R$ with $0<\varepsilon<1$ which satifies
\begin{equation}\label{eq:epsilon_condition}
\frac{n+m}{n}\cdot \frac{1}{(1-\varepsilon)\log m}+m^{-\varepsilon}\le \varepsilon.
\end{equation}
Let $$N=n\big(m-(1-\varepsilon)^2\log m\big).$$
To prove the second statement of the theorem, we shall show that $R(Q_m,Q_n)\le N$.
The first statement is a corollary of this: Indeed, if $m\ge 2^{25}$, then Proposition \ref{prop:epsilon} shows that
inequality (\ref{eq:epsilon_condition}) holds for $\varepsilon=\tfrac{1}{\sqrt{\log m}}$, thus 
$$R(Q_m,Q_n)\le n\big(m-(1-\varepsilon)^2\log m\big) \le n\big(m-(1-2\varepsilon)\log m\big)= n\left(m-\big(1-\tfrac{2}{\sqrt{\log m}}\big)\log m\right).$$

Now we proceed with the proof of the second statement in Theorem \ref{thm:QmQn}, i.e., the bound $R(Q_m,Q_n)\le N$. 
Let $t_\mu=(1-\varepsilon) \log m$ and $t_\eta=\tfrac{n}{m}t_\mu$. Note that $0\leq t_\mu \leq n$ and $0\leq t_\mu \leq m$.
In this proof we consider $t_\mu$-truncated blobs of dimension $m$ and $t_\eta$-truncated blobs of dimension $n$. 
It follows that $N= n(m-t_\mu)+\varepsilon nt_\mu=(n-t_\eta)m+\varepsilon mt_\eta$. 

Fix an arbitrary coloring of $\QQ([N])$. We shall find a blue copy of $Q_m$ or a red copy of $Q_n$ in this coloring.
For this, we show that there is a blue embedding $\phi$ of $\QQ([m])^{t_\mu}$ whose image has volume at most $N-(m-t_\mu)n$,
or a red embedding $\phi'$ of $\QQ([n])^{t_\eta}$ whose image has volume at most $N-(n-t_\eta)m$.
In either case, Lemma \ref{lem:truncatedCompletion} gives the desired monochromatic Boolean lattice.
\\

First, we suppose that there exist disjoint sets $S,T\subseteq[N]$ such that $|S|\le \varepsilon mt_\eta-n$ and $|T|=n$, 
for which the truncated blob $\cB(S;T; t_\eta)$ is monochromatically red.
Then $|S\cup T|\le \varepsilon mt_\eta=N-(n-t_\eta)m$. Thus, part (iii) of  Lemma \ref{lem:truncatedCompletion} implies the existence of a blue copy of $Q_m$ or a red copy of $Q_n$, which completes the proof. Thus, from now on we can assume the following:\\

\noindent
{\bf Property $(\ast)$:}
There is a blue vertex  in every truncated blob $\cB(S;T;  t_\eta)$ with dimension $|T|=n$ and volume $|S\cup T| \le \varepsilon mt_\eta$.\\

Let $\bX$ be an arbitrary subset of $[N]$ of size $m$, and let $\bY=[N]\setminus\bX$.
In the remainder of the proof we construct an $\bX$-good, blue embedding $\phi\colon \QQ(\bX)^{t_\mu}\to\QQ([N])$ such that its image has a small volume.
After stating the complete construction we justify that the defined objects indeed exist.
\\

\subsection{ Construction of a blue embedding $\phi$  of $\QQ(\bX)^{t_\mu}$}

We shall find a blue $\phi(X)$ for each $X\in\QQ(\bX)^{t_\mu}$ layer-by-layer, such that $\phi$ is $\bX$-good, i.e., $\phi(X)\cap \bX= X$ for $X\subseteq \bX$. 
Consider pairwise disjoint subsets $\bY^{(0)}, \bY^{(1)}, \ldots, \bY^{(t_\mu)}$ of $\bY$, where $|\bY^{(0)}|=n$ and $|\bY^{(i)}|=2^{i-1}t_\eta$, for $i=1, \ldots, t_\mu$.
In our construction, we shall make sure that $\phi(X)\cap \bY \subseteq \bY^{(0)}\cup \cdots \cup \bY^{(|X|)}$, 
which will ensure that the volume of the embedded poset $\Vol({\phi}(\QQ(\bX)^{t_\mu})$ is at most $|X| + |\bY^{(0)}| + \cdots + |\bY^{(t_\mu)}|$.
\\

We shall use the notation $Y_X$ to denote $\phi(X)\cap\bY$, i.e., $\phi(X) = X\cup Y_X$. 
To guarantee that the defined function is an embedding, we only need to verify that for any $X_1\subseteq X_2$, we have $Y_{X_1} \subseteq Y_{X_2}$. Indeed, 
since $\phi(X)\cap \bX = X$, for incomparable $X_1$ and $X_2$ the images $\phi(X_1)$ and $\phi(X_2)$ are incomparable.
\\

\begin{figure}[h]
\centering
\includegraphics[scale=0.58]{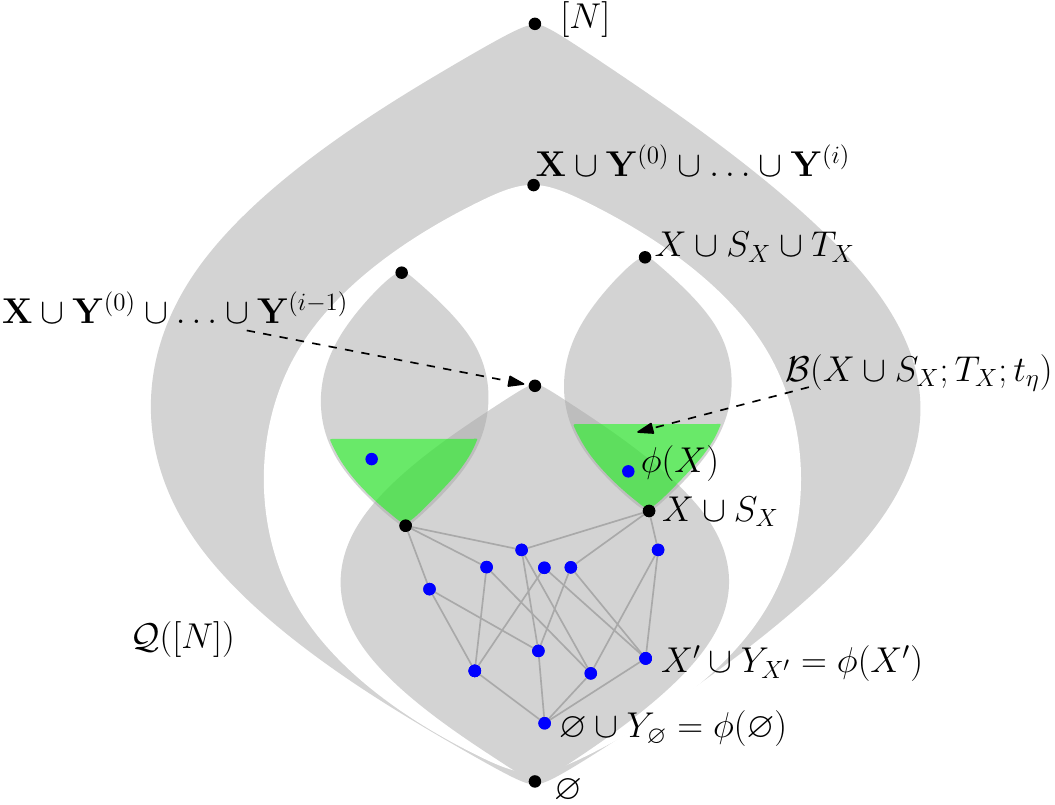}
\caption{The vertex $\phi(X)$ in $\cB(X\cup S_X;T_X;t_\eta)$ for $i=3$ and $|\bX|=4$}
\label{fig:QmQn}
\end{figure}

Let $\phi(\varnothing) $ be a blue vertex in $\cB(\varnothing; \bY^{(0)}; t_\eta)$, that exists by property $(\ast)$. 
For embedding the $i$-th layer, we define a blob whose variable set uses a set $\bY^{(i)}$ of new ground elements as well as ground elements of $\bY$ that were already used by the embedding of previous layers.
Let $i\geq 1$. Assume that we constructed $\phi(X)$ for every $X\subseteq \bX$, $|X|\leq i-1$, such that 
 $\phi(X) \cap \bY\subseteq \bY^{(0)} \cup  \cdots \cup \bY^{(i-1)}.$  
We shall construct $\phi(X)$ for each $X\subseteq \bX$, $|X|=i$, so let
\begin{eqnarray*}
S_X &= &  \bigcup_{X'\subset X}\phi(X') \cap  \bY= \bigcup_{X'\subset X} Y_{X'},\\
T_X&  \subseteq & (\bY^{(0)} \cup \cdots \cup  \bY^{(i)}) \setminus S_X,  ~~~|T_X|=n. 
\end{eqnarray*}

\noindent
Let $\phi(X)$ be a blue vertex in  $\cB(X\cup S_X;  T_X; t_\eta)$,  that exists by $(\ast)$. Let $Y_X=\phi(X)\cap\bY$.\\

We run this procedure for all $i\leq t_\mu$. 
For any two $X_1,X_2\in \QQ(\bX)^{t_\mu}$ with $X_1\subseteq X_2$, it follows from the construction that $Y_{X_1} \subseteq Y_{X_2}$. 
Thus indeed, we have an embedding of $\QQ(\bX)^{t_\mu}$.
\\

\subsection{Verification that $\phi$ is well defined and has small image volume}

\noindent
We need to make sure that the sets $\bY^{(i)}$ and $T_X$ exist, i.e., that the sets from which these are selected as subsets are large enough.  
To see that $\bY^{(i)}$ exists for $i\leq t_\mu$, it is sufficient to verify that $|\bY^{(0)} \cup \cdots \cup  \bY^{(t_\mu)}| \leq |\bY|$.
Note that for $i\geq  1$, 

\begin{equation}\label{eq:Y2}
|\bY^{(0)} \cup \cdots \cup  \bY^{(i)}| = n + \sum_{j=1}^{i} 2^{j-1} t_\eta = n+ (2^i-1) t_\eta.
\end{equation} Recalling that $t_\mu=(1-\varepsilon) \log m$ and $t_\eta=\tfrac{n}{m}t_\mu$, we have 
\begin{eqnarray}
|\bY^{(0)} \cup \cdots \cup  \bY^{(t_\mu)}|  & = & n+ (2^{t_\mu}-1) t_\eta   \nonumber \\
&\le& n+m^{1-\varepsilon}\frac{n}{m}t_\mu \nonumber \\
&=&  \left(\frac{1}{t_\mu} +m^{-\varepsilon}\right)nt_\mu.  \nonumber
\end{eqnarray}
We picked $\varepsilon$ such that $\tfrac{n+m}{n}\cdot \tfrac{1}{(1-\varepsilon)\log m}+m^{-\varepsilon}\le \varepsilon$,
which implies $\tfrac{1}{t_\mu} + \tfrac{m}{nt_\mu} +m^{-\varepsilon}\le \varepsilon$. Thus,

\begin{eqnarray}
|\bY^{(0)} \cup \cdots \cup  \bY^{(t_\mu)}|  & \leq&  \left(\varepsilon  - \frac{m}{nt_\mu}\right)nt_\mu \nonumber\\
&= & \varepsilon nt_\mu - m \nonumber\\
&= &N- n(m-t_\mu) -m   \label{eq:Y1}\\
&\le & N-m \nonumber\\
&=&  |\bY|  \label{eq:Y}.
\end{eqnarray}
It follows from (\ref{eq:Y}) that $\bY^{(i)}$, $i\leq t_\mu$, exists.
\\

Next, we shall show that $T_X$ exists for every $X\subseteq \bX$, $|X|=i$, with $i\in[t_\mu]$. 
For that, it is sufficient to verify that $|(\bY^{(0)} \cup \cdots \cup  \bY^{(i)}) \setminus \bigcup_{X'\subset X}  Y_{X'}|\ge n$.
Observe that in our construction $\phi(X)$ is chosen in a $t_\eta$-truncated blob, in which every vertex is larger than $S_X=\bigcup_{X'\subset X} Y_{X'}$.
Therefore, $|Y_X \setminus \bigcup_{ {X'\subset X }}  Y_{X'}|\leq  t_\eta$, i.e., we are introducing at most $t_\eta$ ``new'' elements from $\bY$ for $\phi(X)$, compared to the images of proper subsets of $X$. 
If $|X|=i$, then $X$ has $2^i-1$ proper subsets $X'$, and $\phi(X')$ uses as most $t_\eta$ ``new'' elements of $\bY$ compared to its own subsets, so we have that  
\begin{equation}\label{eq:Y3}
|S_X|=\bigg|\bigcup_{X'\subset X}  Y_{X'}\bigg|
= \bigg|\bigcup_{X'\subset X}  \bigg(Y_{X'}\setminus  \bigcup_{X''\subset X'} Y_{X''}\bigg)\bigg| \leq (2^i-1) t_\eta.
\end{equation}
Using (\ref{eq:Y2}) and (\ref{eq:Y3}) we have  
$$\bigg|  (\bY^{(0)} \cup \cdots \cup  \bY^{(i)}) \setminus \bigcup_{X'\subset X}  Y_{X'}\bigg| \geq  (n+ (2^i-1)t_\eta) -  (2^i-1) t_\eta =n, $$
so we can select a set $T_X$ of size $n$  from $(\bY^{(0)} \cup \cdots \cup  \bY^{(i)}) \setminus \bigcup_{X'\subset X}  Y_{X'}$.
\\

\subsection{Completion of the proof}

Finally, we consider the volume of $\phi(\QQ(\bX)^{t_\mu})$ and obtain the following bound using (\ref{eq:Y1}):
$$\Vol (\phi( \QQ(\bX)^{t_\mu}))\leq   |\bX| + |\bY^{(0)} \cup \cdots \cup \bY^{(t_\mu)}|  \leq m+  (N- n(m-t_\mu) -m)  = N-n(m-t_\mu).$$
We conclude that $\phi$ is a blue $\bX$-good embedding of $\QQ(\bX)^t_\mu$ with $\Vol(\phi( \QQ(\bX)^{t_\mu}))\leq N - (m-t_\mu)n$. 
Thus, by part (ii) of Lemma \ref{lem:truncatedCompletion} with $t=t_\mu$, we have a blue copy of $Q_m$ or a red copy of $Q_n$.
\end{proof}

\bigskip

\section{Diagonal Ramsey numbers for $D_n$ and $V_n$}\label{sec:DnVn}

\subsection{Diamonds}

Recall that the \textit{Sperner number} $\alpha(n)$ is the smallest $N$ such that $Q_N$ contains an antichain of size $n$,
and Sperner \cite{Sperner} showed that $\binom{\alpha(n)}{\lfloor\alpha(n)/2\rfloor}\ge n$.

\begin{proof}[Proof of Theorem \ref{thm:DnDn}]
For the lower bound, color the Boolean lattice $\QQ^1:=\QQ([2\alpha(n)-1])$ such that $Z\in\QQ^1$ is red if $|Z|<\alpha(n)$ and blue  if $|Z|\ge \alpha(n)$.
Assume that there is a red copy $\cD$ of $D_n$ in this coloring with maximal vertex $Y$. 
Then $|Y|<\alpha(n)$, so $\cD$ is contained in the sublattice $\{Z\in\QQ^1:~ Z \subseteq Y\}$. 
We know that $\cD$ contains an antichain on $n$ vertices, but by definition of $\alpha(n)$ there is no antichain on $n$ vertices in $\{Z\in\QQ^1:~ Z \subseteq Y\}$, a contradiction.
Similarly, we obtain that there is no blue copy of $D_n$.
\\

In order to bound $R(D_n,D_n)$ from above, let $N= \alpha(n)+\alpha(2n-1)$ and consider an arbitrary coloring of $\QQ^2:=\QQ([N])$.
Let $$\cS=\Big\{Z\in\QQ^2:~ |Z|=\lfloor\alpha(n)/2\rfloor\Big\} \text{ and } \cT=\Big\{Z\in\QQ^2:~ |Z|=N-\lfloor\alpha(n)/2\rfloor\Big\}.$$
We distinguish two cases.
\\

\textbf{Case 1:} At least one of $\cS\cup\{\varnothing\}$ or $\cT\cup\{[N]\}$ is not monochromatic.\\
Suppose $\cT\cup\{[N]\}$ is not monochromatic. Let $Y\in\cT$ such that $Y$ has a different color than the vertex $[N]$, see Figure \ref{fig:DnDn}.
Let $\cS'=\{Z\in\cS:~ Z\subseteq Y\}$. 

If $\cS'\cup\{\varnothing\}$ is monochromatic, then one of $Y$ or $[N]$ has the same color as $\cS'$.
Note that $|\cS'|= \binom{|Y|}{\lfloor\alpha(n)/2\rfloor}\ge \binom{\alpha(n)}{\lfloor\alpha(n)/2\rfloor}\ge n$, where the last inequality follows from Sperner's theorem.
This implies that the vertices $\cS'\cup\{\varnothing,Y, [N]\}$ contain a monochromatic copy of $D_n$.

If $\cS'\cup\{\varnothing\}$ is not monochromatic, then consider $X\in\cS'$ such that $X$ has a different color than the vertex $\varnothing$.
Note that $X\subset Y$.
The sublattice $\{Z\in\QQ^2:~ X\subseteq Z\subseteq Y\}$ has dimension $|Y|-|X|\ge N-\alpha(n)\ge \alpha(2n-1)$.
This implies that there is an antichain $\cA$ on $2n-1$ vertices, such that for every $Z\in\cA$, $X\subseteq Z\subseteq Y$. Note that $X,Y\notin\cA$.
In $\cA$ we find $n$ vertices with the same color, say without loss of generality red. Then these $n$ vertices together with the red vertex among $\varnothing$ and $X$ and the red vertex among $Y$ and $[N]$ form a red copy of $D_n$.
\\

\begin{figure}[h]
\centering
\includegraphics[scale=0.58]{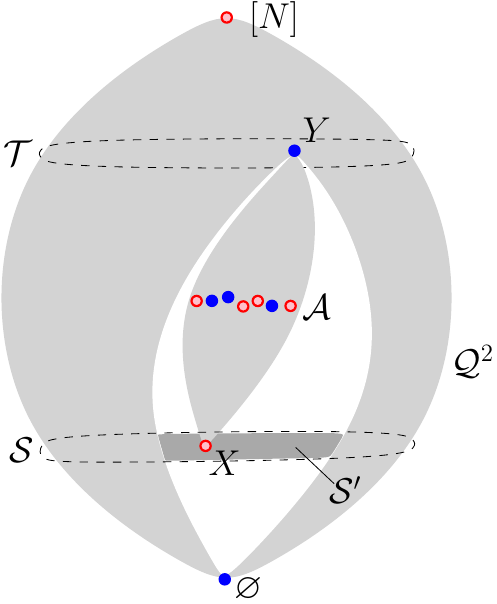}
\caption{Setting in Case 1 if $\cS'\cup\{\varnothing\}$ is not monochromatic}
\label{fig:DnDn}
\end{figure}

\textbf{Case 2:} Both $\cS\cup\{\varnothing\}$ and $\cT\cup\{[N]\}$ are monochromatic.\\
If $\cS\cup\{\varnothing\}$ and $\cT\cup\{[N]\}$ have the same color, then $\cS\cup\{\varnothing, [N]\}$ contains a monochromatic copy of $D_n$ because 
$|\cS|=\binom{N}{\lfloor\alpha(n)/2\rfloor}\ge \binom{\alpha(n)}{\lfloor\alpha(n)/2\rfloor}\ge n$.
So suppose that $\cS\cup\{\varnothing\}$ and $\cT\cup\{[N]\}$ have distinct colors, say $\cS\cup\{\varnothing\}$ is red and $\cT\cup\{[N]\}$ is blue.
Fix the vertex $X=[\alpha(n)]\in\QQ^2$.
If $X$ has the same color as $\cS\cup\{\varnothing\}$, i.e.\ red, then let $\cS''=\{Z\in\cS:~ Z\subseteq X\}$.
Note that $|\cS''|\ge \binom{|X|}{\lfloor\alpha(n)/2\rfloor}\ge n$. Then $\cS''\cup\{\varnothing,X\}$ contains a red copy of $D_n$.
If $X$ has the same color as $\cT\cup\{[N]\}$, we find a monochromatic copy of $D_n$ by a similar argument.
\end{proof}

\subsection{Forks}

Let $n,N\in\N$ such that $N\ge \alpha(n)$.
Recall that
$$\beta(N,n)=\min\big\{\beta\in\N: ~ \binom{N}{\beta}\ge n\big\}\ \text{ and }\ N^*(n)=\max\big\{N\ge\alpha(n): ~ N-\beta(N,n)< \alpha(n)\big\},$$
as illustrated in Figure \ref{fig:Nstar}.
Both $\beta(N,n)$ and $N^*(n)$ are well-defined.

\begin{figure}[h]
\centering
\includegraphics[scale=0.58]{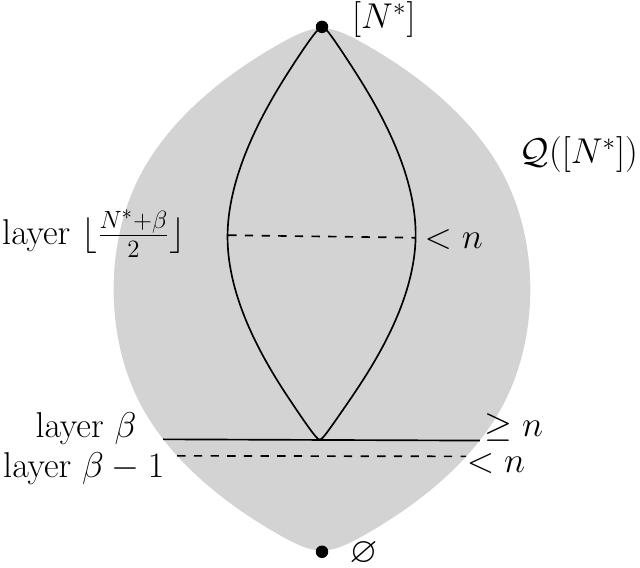}
\caption{Setting for $N^*$ and $\beta=\beta(N^*,n)$}
\label{fig:Nstar}
\end{figure}

\begin{proof}[Proof of Theorem \ref{thm:VnVn}]
Let $N^*=N^*(n)$. First, we show the lower bound $R(V_n,V_n)>N^*$.
We construct a coloring of $\QQ^1:=\QQ([N^*])$ which contains no monochromatic $V_n$.
Color each vertex $Z$ with $|Z|<\beta(N^*,n)$ in red and all remaining vertices in blue.
Then there is no red antichain of size $n$, so in particular there is no red copy of $V_n$.
Assume towards a contradiction that there is a blue copy of $V_n$ with minimal vertex $X$.
Note that $|X|\ge \beta(N^*,n)$, so the sublattice $\{Z\in\QQ^1:~X\subseteq Z \subseteq [N^*]\}$ has dimension at most $N^*-\beta(N^*,n)< \alpha(n)$,
thus there is no blue antichain of size $n$, and in particular no blue copy of $V_n$, a contradiction.
\\

For the upper bound, we define $N_{+}$ to be the smallest integer such that $N_{+}-\beta(N_{+},n)\ge \alpha(2n-1)$.
In order to show that $R(V_n,V_n)\le N_{+}$, we consider an arbitrary coloring of $\QQ^2:=\QQ([N_{+}])$. 
We shall find a monochromatic copy of $V_n$.
Without loss of generality, the vertex $\varnothing$ is red.
We know that layer $\beta_+:=\beta(N_{+},n)$ contains at least $n$ vertices.
If layer $\beta_+$ is red, we find a red copy of $V_n$, so suppose that there exists a blue vertex $X$ with $|X|=\beta_+$.
Then the sublattice $\QQ^3:=\{Z\in\QQ^2:~X\subseteq Z\subseteq [N_{+}]\}$ has dimension $N_{+}-\beta_+\ge \alpha(2n-1)$. 
Thus we find an antichain $\cA$ of size $2n-1$ in $\QQ^3$. Note that $X\notin \cA$.
Each vertex in $\cA$ is either red or blue, so there are $n$ vertices of the same color in this antichain, which together with one of $X$ or $\varnothing$ form a monochromatic copy of $V_n$ as desired. Therefore, $R(V_n,V_n)\le N_+$.
\\

We claim that $N_+\le N^*+3$. To prove this, we first show that $\beta(N^*+1,n)\ge \beta(N^*+3,n)$.
Recall that $N^*\ge \alpha(n)$, so $\beta(N^*+1,n) \le \lfloor\alpha(n)/2\rfloor\le N^*/2$.
This implies that  
$$
\binom{N^*+3}{\beta(N^*+1,n)}\ge \binom{N^*+1}{\beta(N^*+1,n)}\ge n,
$$ 
thus indeed $\beta(N^*+1,n)\ge\beta(N^*+3,n)$.
By definition of $N^*(n)$ we know that $(N^*+1)-\beta(N^*+1,n)\ge \alpha(n)$, so
$$(N^*+3)-\beta(N^*+3,n)\ge (N^*+1)+2-\beta(N^*+1,n)\ge \alpha(n)+2\ge \alpha(2n-1).$$
Recall that $N_+$ is minimal such that $N_{+}-\beta(N_{+},n)\ge \alpha(2n-1)$, so $N_+\le N^*+3$.
This concludes the proof of the upper bound. 
We remark that if $\alpha(n)+1\ge \alpha(2n-1)$, a similar argument provides $N_+\le N^*+2$.
\\

In the Appendix we show that $N^*(n)= (d+o(1))\log n$, where $d=\frac{1}{1-c}$ and  $c$ is the unique real solution of $\log\big(c^{-c}(1-c)^{c-1}\big)=1-c$, i.e.\ $d\approx 1.29$.

\end{proof}

\bigskip

\section{Upper bound on Ramsey number   $\Rw (Q_n,Q_n)$}\label{sec:weakQnQn}

\begin{proof} [Proof of  Theorem \ref{thm:weakQnQn}]
Consider an arbitrary coloring of the Boolean lattice $\QQ([N])$ where $N=0.96n^2$.
Our goal is to find a monochromatic weak copy of $Q_n$. While an induced copy of $Q_n$ has a rigid structure, there are many non-isomorphic weak copies of $Q_n$. 
In $\QQ([N])$ we shall find a member of a class $\cP(n,s, t)$ of special weak copies of $Q_n$.

\subsection{Definition of $\cP(n,s,t)$}
 In this section we write $P'<P''$ for posets $P'$ and $P''$ if any element of a poset $P'$ is less than any element of a poset $P''$.  
 We define the class $\cP(n,s,t)$ of posets, see Figure \ref{fig:QnQn_weak} (a),  such that each member of this class is of the form
  $$P_0 \cup  \dots \cup P_{s-1}\cup  Q_s^t \cup  P'_{t+1} \cup \dots \cup P'_n,$$
 where 
 \begin{itemize}
 \item{} $P_i$ is an arbitrary poset with $|P_i|=\binom{n}{i}$, $i\in \{0, \ldots, s-1\}$,
 \item{} $P'_j$ is an arbitrary poset with $|P'_j|=\binom{n}{n-j}=\binom{n}{j}$,  $j\in \{t+1, \ldots, n\}$, 
 \item{} $Q_s^t$ is an (induced) copy of an $(s,t)$-truncated $Q_n$, i.e., consists of layers $s,\dots, t$ of $Q_n$,
 \item{} $P_0< P_1< \cdots < P_{s-1} < Q_s^t < P'_{t+1} < \cdots <P'_n$. \end{itemize}
 
Here, if $s=0$ or $t=n$, we use the convention that $P_0 \cup  \cdots \cup P_{s-1}=\varnothing$ or $P'_{t+1} \cup \cdots \cup P'_n=\varnothing$, respectively.
Observe that every member of $\cP(n,s,t)$ is indeed a weak copy of $Q_n$, where layer $i$ of $Q_n$ corresponds to $P_i$, for $i\in \{0, \ldots, s-1\}$, 
layer $j$ of $Q_n$ corresponds to $P'_{j}$, for $j\in \{t+1, \ldots, n\}$ and the remaining layers are contained in the \textit{middle part} $Q_s^t$.


\begin{figure}[h]
\centering
\includegraphics[scale=0.58]{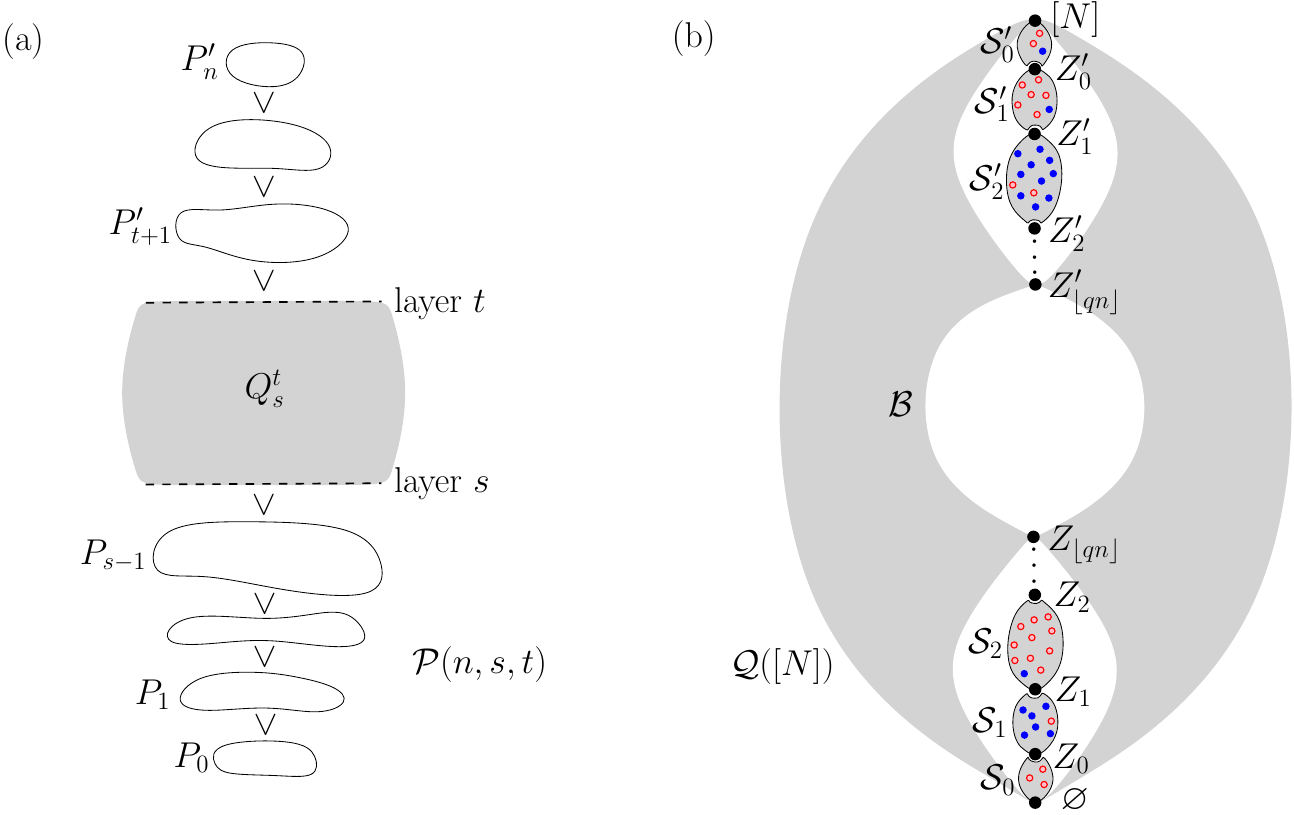}
\caption{(a) A $\cP(n,s,t)$ for $s=4$ and $t=n-3$, (b) Sausage chain in $\QQ([N])$}
\label{fig:QnQn_weak}
\end{figure}

\subsection{Construction of a sausage chain in $\QQ([N])$}
Let $N=0.96n^2$.
For our proof we need to define a constant $q$ that satisfies certain properties. For  $0<p<1$, the \textit{binomial entropy function} is defined as  
$H(p)=-\big(p\log p + (1-p)\log(1-p)\big),$ where $\log$ is base $2$. 
Let $q$ be the real number, $0<q<1/2$, which minimizes $(1-q) + 2\int_{0}^{q} H(s) ds$.  It can be verified analytically or numerically that such a $q$ satisfies $H(q)=1/2$, i.e., $0.11<q<0.111$, and $\int_{0}^{q} H(s) ds\le 0.033$.
In particular, for sufficiently large $n$ we have that  
\begin{equation}\label{eq:minimize}
(1-q)n^2 + 2n^2 \int_{0}^{q} H(s)  \operatorname{d}\!s \le 0.956n^2 \le N-\varepsilon n^2,
\end{equation}
for some constant $\varepsilon>0$.

Next we define a {\it sausage chain} in $\QQ([N])$, see Figure \ref{fig:QnQn_weak} (b). 
Let $$Z_0\subset Z_1 \subset \dots \subset Z_{\qn} \subset Z'_{\qn}\subset Z'_{\qn-1}\subset \dots \subset Z'_0$$ be vertices in $\QQ([N])$ such that for $0\leq i\leq \qn$,
$$|Z_i|=\sum_{j=0}^{i} \left(\left\lceil\log\binom{n}{j}\right\rceil+1\right) \quad\text{ and }\quad |Z'_i|=N-\sum_{j=0}^{i} \left(\left\lceil\log\binom{n}{j}\right\rceil+1\right).$$
We will argue later that $|Z_{\qn}|\le|Z'_{\qn}|$, which implies that these vertices are well-defined.
We define subposets $\cS_i$ and $\cS'_i$, which we refer to as \textit{sausages}. Let $\cS_0=\{X\in\QQ([N])\colon X\subset Z_0\}$, and
$$\cS_i=\{X\in\QQ([N])\colon Z_{i-1}\subseteq X \subset Z_i\}.$$
Similarly, let \textit{sausages} $\cS'_0=\{X\in\QQ([N])\colon Z'_0\subset X\}$, and 
$$\cS'_i=\{X\in\QQ([N])\colon Z'_{i}\subset X \subseteq Z'_{i-1}\}.$$
Moreover, let 
$$\cB=\{X\in\QQ([N])\colon Z_{\qn}\subseteq X \subseteq Z'_{\qn}\}.$$
The subposet $\cB$ is isomorphic to a Boolean lattice of dimension $|Z'_{\qn}|-|Z_{\qn}|$.
Note that $$\cS_0<\cS_1<\dots<\cS_{\qn}<\cB<\cS'_{\qn}<\dots<\cS'_0.$$
The subposet $\cS_0\cup\dots \cS_{\qn}\cup\cB\cup\cS'_{\qn}\cup\dots\cup\cS'_0$ of $\QQ([N])$ is referred to as \textit{sausage chain}.
Note that the sausage chain is well-defined if all vertices $Z_i,Z'_i$, $i\in\{0,\dots,\qn\}$, exist. 
Note also that $|\cS_0|<\cdots < |\cS_{\qn}|$ and $|\cS'_{\qn}|>\cdots >|\cS_0|$.

\subsection{Finding a member of $\cP(n,s,t)$ in the sausage chain}
We shall find a monochromatic member of $\cP(n,s,t)$ for some $s$ and $t$ depending on the coloring, such that the middle part $Q_s^t$ of $\cP(n,s,t)$ is embedded into $\cB$, each $P_i$ is embedded into its own $\cS_\ell$, and each $P'_i$ is embedded in its own  $\cS'_{\ell'}$ for some $\ell$, $\ell'$.
\\

Assume, without loss of generality, that among all sausages $\cS_0, \ldots, \cS_{\qn}, \cS'_0, \ldots, \cS'_{\qn}$, most sausages have majority color red,
then at least $\QN$ sausages have this majority color.
Assume further, that there are $s$ sausages among  $\cS_0, \ldots, \cS_{\qn}$ with majority color red,
which we denote by $\cS_{i_0},\dots,\cS_{i_{s-1}}$, $i_0<\dots<i_{s-1}$. Note that possibly $s=0$.
Since $i_0 \ge 0$, we see that $i_1\ge 1$, and similarly for any $j\in\{0,\dots,s-1\}$, $i_j\ge j$.
For $j\in\{0, \ldots, s-1\}$, choose $P_j$ arbitrarily such that $P_j\subseteq \cS_{i_j}$,  $|P_j|=\binom{n}{j}$, and $P_j$ is red.
These subposets are well-defined since for any $i\in\{0,\dots,\qn\}$, $|\cS_i|=2^{|Z_i|-|Z_{i-1}|}-1=2^{\lceil\log\binom{n}{i}\rceil+1}-1\ge 2\binom{n}{i}-1,$
so in particular there are at least $ \binom{n}{i_j}\geq \binom{n}{j}$ distinct red vertices in $\cS_{i_j}$.
\\

Similarly, we find $\QN-s$ sausages among $\cS'_0, \ldots, \cS'_{\qn}$ with majority color red, denoted by $\cS'_{i'_0},\dots,\cS'_{i'_{\QN-s-1}}$, $i'_0<\dots<i'_{\QN-s-1}$. Here, it is possible that $\QN-s=0$.
For $j\in\{0,\dots,\QN-s-1\}$, choose $P'_{n-j}$  arbitrarily, such that $P'_{n-j}\subseteq \cS'_{i'_j}$, $|P'_{n-j}|=\binom{n}{n-j}$, and $P'_{n-j}$ is red.
With a similar argument as before, we see that there are indeed at least $\binom{n}{n-j}$ distinct red vertices in $\cS'_{i'_j}$.
\\

So far we selected $P_j$ for $j\in\{0, \ldots, s-1\}$ and $P'_j$ for $j\in\{t+1,\dots,n\}$, where $t=n-\QN+s$.
It remains to verify that $Q_s^t$ is contained in $\cB$. For that we shall show that the dimension of $\cB$ is large enough to apply Lemma \ref{lem:truncatedCompletion} (iv).  
Recall that $H(p)=-\big(p\log p + (1-p)\log(1-p)\big)$.
Stirling's formula implies that $\log\binom{N}{pN}=\big(1+o(1)\big)H(p)N$ for any positive integer $N$ and $0<p<1$. Thus
\begin{align}
{\rm dim} (\cB) & =  |Z'_{\qn}|-|Z_{\qn}|\nonumber\\
&=N-2\sum_{i=0}^{\qn} \left(\left\lceil\log\binom{n}{i}\right\rceil+1\right)\nonumber\\
&\ge N-4n-2\sum_{i=1}^{\qn} \log\binom{n}{i}\nonumber\\
&\ge N-4n-\big(2+o(1)\big)n\sum_{i=1}^{\qn} H\left(\frac{i}{n}\right). \label{eq:H_in}
\end{align}

Since $H$ is an increasing and continuous function on the interval $(0,1/2)$ and is bounded by $1$, we have
$$
\sum_{i=1}^{\qn} H\left(\frac{i}{n}\right)\le \int_{1}^{qn+1} H\left(\frac{t}{n}\right)\,\operatorname{d}\!t=\int_{1/n}^{q+1/n}H(s)n\,\operatorname{d}\!s\le n \int_{0}^{q}H(s)\,\operatorname{d}\!s + 1,
$$
Thus, using (\ref{eq:H_in}) and recalling the bound on $N$ from (\ref{eq:minimize}), we have 
\begin{eqnarray*}
{\rm dim} (\cB) &\ge &  N- 4n - \big(2+o(1)\big) \left(n^2\int_{0}^{q} H(s) \operatorname{d}\!s +n\right)\\
& \ge & N-2n^2\int_{0}^{q} H(s) \operatorname{d}\!s -o(n^2)\\
& \ge & \left[ (1-q)n^2 + 2n^2 \int_{0}^{q} H(s)  \operatorname{d}\!s +\varepsilon n^2\right] -2n^2\int_{0}^{q} H(s) \operatorname{d}\!s -o(n^2)\\
& \geq & (1-q)n^2+2n = (n-qn+2)n.
\end{eqnarray*}
In particular, $|Z'_{\qn}|-|Z_{\qn}|\ge 0$, which implies that $Z_{i}$ and $Z'_{i}$, $i\in\{0,\dots,\qn\}$, are well-defined.
Since ${\rm dim} (\cB) \ge  (n-qn+2)n \ge (t-s+2)n$ for $t=n-\QN+s$, 
Lemma \ref{lem:truncatedCompletion} (iv) provides that $\cB$ contains either a blue (induced) copy of $Q_n$ and we are done, or a red $Q_s^t$.
If there is a red $Q_s^t$ in $\cB$, we conclude that the sausage chain contains a red member of $\cP(n, s, t)$, and thus a red weak copy of $Q_n$.
\end{proof}

\section{Concluding remarks}
In this paper, we first studied induced poset Ramsey numbers and 
showed that for $n\ge m$ and $0<\varepsilon<1$ with $\tfrac{n+m}{n}\cdot \tfrac{1}{(1-\varepsilon)\log m}+m^{-\varepsilon}\le \varepsilon$, 
$$R(Q_m,Q_n)\le n\left(m-(1-\varepsilon)^2\log m\right).$$
When applying this result for specific $\varepsilon$, there is a trade-off between the best Ramsey bound and the smallest value of $m$ for which the bound holds.
Our main result claims that $R(Q_m,Q_n)\le n\left(m-\big(1-\tfrac{2}{\sqrt{\log m}}\big)\log m\right)$ for $2^{25}\le m\le n$.
In addition, for $1024\le m\le n$ or $32\le m\le \frac{n}{32}$, one could obtain the upper bound $R(Q_m,Q_n)\le n\left(m-\tfrac14\log m\right)$, using $\varepsilon=\frac12$. 
\\

Theorem \ref{thm:QmQn} is an improvement of the basic upper bound, Lemma \ref{lem:blob_restated}, by a superlinear additive term and a step towards the following conjecture raised by Lu and Thompson \cite{LT}.

\begin{conjecture}[\cite{LT}]\label{conj:LT}
For $n\ge m$ and sufficiently large $m$, $R(Q_m,Q_n)=o(n^2)$.
\end{conjecture}
\noindent We propose a stronger conjecture:
\begin{conjecture}\label{conj:LT}
For any $\varepsilon>0$, there is a large enough $m_0$ such that for any two $m,n\in\N$ with $n\ge m \ge m_0$,
$$R(Q_m,Q_n)\le n\cdot m^{\varepsilon}.$$
\end{conjecture}

\noindent We remark that for small $m$, the authors conjectured the following in \cite{QnV}.
\begin{conjecture}[\cite{QnV}]
For fixed $m\in\N$ and large $n\in\N$ , $R(Q_m,Q_n)=n+o(n)$.
\end{conjecture}

In the last part of this paper we discussed weak poset Ramsey numbers and improved the previously known upper bound $\Rw(Q_n,Q_n)\le R(Q_n,Q_n)\le n^2-o(n^2)$ to $\Rw(Q_n,Q_n)\le 0.96n^2$. 
It is still open whether or not the weak poset Ramsey number is significantly smaller than the (induced) poset Ramsey number.
The second author suggests the following conjecture.
\begin{conjecture}
For any $n\in\N$,
$R(Q_n,Q_n)-\Theta(n)\le\Rw(Q_n,Q_n)\le R(Q_n,Q_n)$.
\end{conjecture}

\bigskip


\noindent \textbf{Acknowledgments:}~\quad   Research was partially supported by DFG grant FKZ AX 93/2-1.


\section{Appendix - Calculation for the proof of Theorem \ref{thm:VnVn}}
We shall find $d$ such that $N^*= (d+o(1))\log n$, where 
$$N^*=\max\big\{N\ge\alpha(n): ~ N-\beta(N,n)< \alpha(n)\big\}\ \text{ and }\ \beta(N,n)=\min\big\{\beta\in\N: ~ \binom{N}{\beta}\ge n\big\}.$$
For arbitrary $N\in\N$ and $q$ with $0<q<1$, Stirling's formula provides that
\begin{equation}\label{eq:Hq}
\log\binom{N}{qN}=-\big(1+o(1)\big)N\big(q\log q + (1-q)\log(1-q)\big)=\big(1+o(1)\big)H(q)N,
\end{equation}
where $H(q)=-\big(q\log q + (1-q)\log(1-q)\big)$ is the \textit{binary entropy function}.
Let $c$ be the unique solution of $1-c=H(c)$, i.e.\ $c\approx0.2271$.
We shall show that $N^*=\big(\tfrac{1}{1-c}+o(1)\big)\log n$.
Let $q$ such that $qN^*=\beta(N^*,n)$, and let $q'$ such that $q'(N^*+1)=\beta(N^*+1,n)$.
The definition of $\beta$ implies that
\begin{equation}\label{eq:beta}
\binom{N^*}{qN^*-1}<n\le\binom{N^*}{qN^*}\qquad\text{and} \qquad \binom{N^*+1}{q'(N^*+1)-1}<n\le\binom{N^*+1}{q'(N^*+1)},
\end{equation}
By the definition of $N^*(n)$, we know that
\begin{equation}\label{eq:Nstar}
(1-q)N^*<\alpha(n)\le (1-q')(N^*+1).
\end{equation}

In the following, $o(1)$ always refers to the asymptotic behaviour for large $n$, so equivalently for large $N^*$ as $\alpha(n)\le N^*\le 2\alpha(n)$.
Recall that 
\begin{equation}\label{eq:alpha}
\alpha(n)=\big(1+o(1)\big)\log n.
\end{equation}
\\

We shall compute $q$. We label each step of our computation by an inequality from (\ref{eq:Hq}) to (\ref{eq:alpha}) which is being used. 
For example,  `$\stackrel{(\ref{eq:Hq})}{=}$' means that the equality holds because of (\ref{eq:Hq}).
To highlight the idea of the upcoming computation, we give a one-line proof, where some steps are not yet justified:
$$(1-q)N^*\stackrel{(\ref{eq:Nstar})}{\approx}\big(1+o(1)\big)\alpha(n)\stackrel{(\ref{eq:alpha})}{=}\big(1+o(1)\big)\log n
\stackrel{(\ref{eq:beta})}{\approx} \big(1+o(1)\big)\log \binom{N^*}{qN^*} \stackrel{(\ref{eq:Hq})}{=} \big(1+o(1)\big)H(q)N^*,$$
which would imply $q=\big(1+o(1)\big)c$ where $c$ is the unique solution of $1-c=H(c)$. However, some steps in the above computation require significant additional argumentation, which we give in the following.
\\

Observe that
$$(1-q)N^*\stackrel{(\ref{eq:Nstar})}{<}\alpha(n)\stackrel{(\ref{eq:alpha})}{=}\big(1+o(1)\big)\log n
\stackrel{(\ref{eq:beta})}{\le} \big(1+o(1)\big)\log \binom{N^*}{qN^*} \stackrel{(\ref{eq:Hq})}{=} \big(1+o(1)\big)H(q)N^*.$$
Thus, $1-q\le \big(1+o(1)\big)H(q)$, which implies that $q\le  \big(1+o(1)\big) c$.
Next we bound $q'$ from below. We see that
$$(1-q')(N^*+1)\stackrel{(\ref{eq:Nstar})}{\ge}\alpha(n)\stackrel{(\ref{eq:alpha})}{=}\big(1+o(1)\big)\log n
\stackrel{(\ref{eq:beta})}{>} \big(1+o(1)\big)\log \binom{N^*+1}{q'(N^*+1)-1}.$$
We shall show that $\log \binom{N^*+1}{q'(N^*+1)-1}\ge \big(1+o(1)\big)H(q')(N^*+1)$. For that, we first need a rough lower bound on $q'$.
\\

We know from (\ref{eq:Nstar}) that $N^*-qN^*\le \alpha(n)-1$. Note that $qN^*=\beta(N^*,n)\le \alpha(n)$, so $N^*+1\le qN^*+ \alpha(n)\le 2\alpha(n)$. Then
$$\binom{N^*+1}{\tfrac{1}{16}(N^*+1)}\le \binom{2\alpha(n)}{\tfrac{1}{8}\alpha(n)} \stackrel{(\ref{eq:Hq})}{=}
\left(\frac{2^2}{(\tfrac{1}{8})^{1/8}(\tfrac{15}{8})^{15/8}}\right)^{(1+o(1))\alpha(n)}\le 1.6^{(1+o(1))\log n}<n,$$
thus $q'\ge \tfrac{1}{16}$. This bound implies that
\begin{align*}
\binom{N^*+1}{q'(N^*+1)-1}&=\frac{q'(N^*+1)}{(1-q')(N^*+1)+1}\binom{N^*+1}{q'(N^*+1)}\\
&\ge \frac{q'}{2-q'}\binom{N^*+1}{q'(N^*+1)}\ge \frac{1}{31}\binom{N^*+1}{q'(N^*+1)}.
\end{align*}
Thus, 
$$\log \binom{N^*+1}{q'(N^*+1)-1}\ge -\log(31) + \log \binom{N^*+1}{q'(N^*+1)}\stackrel{(\ref{eq:Hq})}{=}\big(1+o(1)\big)H(q')(N^*+1).$$
Therefore, $1-q'\ge \big(1+o(1)\big)H(q')$, which implies that $q'\ge\big(1+o(1)\big) c$. 
\\

We obtain that
$$\alpha(n)\stackrel{(\ref{eq:Nstar})}{\le} (1-q')(N^*+1)\le (1+o(1))(1- c)(N^*+1)\le  \big(1+o(1)\big)(1-q)N^*  \stackrel{(\ref{eq:Nstar})}{\le} \big(1+o(1)\big)\alpha(n),$$
thus $N^*=\tfrac{(1+o(1))}{1-c}\alpha(n)=\tfrac{(1+o(1))}{1-c}\log n,$ as desired.

\end{document}